\documentclass[11pt,a4paper]{article}
\usepackage{amsmath}
\usepackage{amsthm}
\usepackage{makeidx}
\usepackage{latexsym}
\usepackage{graphicx}
\usepackage{amssymb}
\usepackage{hyperref}
\usepackage{geometry}
\usepackage{stackrel}

\newtheorem{theorem}{Theorem}

\newtheorem{proposition}{Proposition}
\newtheorem{corollary}{Corollary}
\newtheorem{remark}{Remark}
\newtheorem{definition}{Definition}
\theoremstyle{remark}
\newtheorem{example}{\textbf{Example}}

\usepackage[latin1]{inputenc}

\DeclareRobustCommand{\stirlingtwo}{\genfrac\{\}{0pt}{}}
\def\s{\atopwithdelims[]}

\title{Some combinatorial properties of the Hurwitz series ring}
\author{Stefano Barbero, Umberto Cerruti, Nadir Murru\\ \\
Department of Mathematics G. Peano, University of Turin,\\
Via Carlo Alberto 10, 10123, Torino, ITALY\\ \\
stefano.barbero@unito.it, umberto.cerruti@unito.it, nadir.murru@unito.it}
\date{}

\begin{document}
\maketitle

\begin{abstract}
We study some properties and perspectives of the Hurwitz series ring $H_R[[t]]$, for a commutative ring with identity $R$. Specifically, we provide a closed form for the invertible elements by means of the complete ordinary Bell polynomials, we highlight some connections with well--known transforms of sequences, and we see that the Stirling transforms are automorphisms of $H_R[[t]]$. Moreover, we focus the attention on some special subgroups studying their properties. Finally, we introduce a new transform of sequences that allows to see one of this subgroup as an ultrametric dynamic space.
\end{abstract}

\section{The ring of Hurwitz series, transformations of sequences and automorphisms}

Given a commutative ring with identity $R$, let $H_R[[t]]$ denote the Hurwitz series ring whose elements are the formal series of the kind

$$A(t) := \sum_{n = 0}^{+\infty} \frac{a_n}{n!}t^n,$$

equipped with the standard sum and the binomial convolution as product. Given two formal series $A(t)$ and $B(t)$, the binomial convolution is defined as follows:

$$A(t) \star B(t) := C(t),$$

where

$$c_n := \sum_{h = 0}^n \binom{n}{h}a_hb_{n-h}.$$

The ring of Hurwitz series has been organizationally studied by Keigher \cite{Kei1} and in the recent years it has been extensively studied, see, e.g., \cite{Kei2}, \cite{Liu}, \cite{Gha}, \cite{Ben1}, \cite{Ben2}, \cite{Ben3}.

The Hurwitz series ring is trivially isomorph to the ring $H_R$ whose elements are infinite sequences of elements of $R$, with operations $+$ and $\star$. In the following, when we consider an element $a \in H_R$, we refer to a sequence $(a_n)_{n=0}^{+\infty}=(a_0,a_1,a_2,...)$, $a_i \in R$ for all $i \geq 0$, having exponential generating function (e.g.f.) denoted by $A(t)$.
Clearly for two sequences $a,b \in H_R$ with exponential generating functions $A(t)$ and $B(t)$ respectively, the sequence $c=a\star b$ has e.g.f. $C(t)=A(t)B(t)$. Moreover, fixed any positive integer $n$, we can also consider the rings $H_R^{(n)}$ whose elements are sequences of elements of $R$ with length $n$.

\begin{remark}
The binomial convolution is a commutative product and the identity in $H_R$ is the sequence
$$(1,0,0,...).$$
Moreover, $H_R$ can be also viewed as an $R$--algebra considering the map
$$\pi : R \rightarrow H_R, \quad \pi(r) := (r,0,0,...),$$
for any $r \in R$.
\end{remark}

\begin{proposition}
An element $a \in H_R$ is invertible if and only if $a_0 \in R$ is  invertible, i.e., 
$$H_R^* = \{a \in H_R : a_0 \in R^*\}.$$
\end{proposition}
\begin{proof}
The proof is straightforward.
\end{proof}

Given $a \in H_R^*$, we can recursively evaluate the terms of $b = a^{-1}$. Indeed,  $b_0 = a_0^{-1}$ and for all $n\geq1$
$$b_n = -a_0^{-1}\sum_{h=1}^n \binom{n}{h}a_hb_{n-h},$$
since the equality $a \star b = (1,0,0,...)$ implies $$a_{0}b_{0}=1, \quad
\sum_{h=0}^n\binom{n}{h}a_nb_{n-h}=0, \forall n\geq1.$$
On the other hand, we can find a nice closed form for the elements of $b$  by means of the complete ordinary Bell polynomials \cite{Bell}. First of all, we recall their definition as given in \cite{Port}.
\begin{definition} \label{def:bell}
Let us consider the sequence $x=(x_1,x_2,\ldots)$. The \emph{complete ordinary Bell polynomials} are defined by
$$B_0(x)=1, \quad \forall n\geq 1 \quad B_n(x)=B_n(x_1,x_2,\ldots,x_n)=\sum_{k=1}^nB_{n,k}(x),$$
where $B_{n,k}(x)$ are the \emph{partial ordinary Bell polynomials},  with
$$ B_{0,0}(x)=1, \quad \forall n\geq 1 \quad B_{n,0}(x)=0,\quad \forall k\geq 1 \quad B_{0,k}(x)=0,$$
$$B_{n,k}(x) =B_{n,k}(x_1,x_2,\ldots,x_{n}) =\sum_{\substack {i_1  + 2i_2  +  \cdots  + ni_n = n \\ i_1  + i_2  +  \cdots  + i_n  = k}} \frac{{k!}}{{i_1 !i_2 ! \cdots i_n !}}x_1^{i_1 } x_2^{i_2 }  \cdots x_n^{i_n },$$
or, equivalently,
$$B_{n,k}(x)=B_{n,k}(x_1,x_2,\ldots,x_{n-k+1})=\sum_{\substack {i_1  + 2i_2  +  \cdots  + (n-k+1)i_{n-k+1} = n-k+1 \\ i_1  + i_2  +  \cdots  + i_{n-k+1}  = k}} \frac{{k!}}{{i_1 !i_2 ! \cdots i_{n-k+1} !}}x_1^{i_1 } x_2^{i_2 }  \cdots x_{n-k+1}^{i_{n-k+1} },$$
satisfying the equality
$$ \left(\sum_{n\geq1}x_nz^n\right)^k=\sum_{n \geq k}B_{n,k}(x)z^n. $$
\end{definition}
Then, we introduce the Invert transform (see, e.g., \cite{Ber} for a detailed survey).

\subsubsection*{The Invert transform}
The Invert transform $\mathcal I$ maps a sequence $a=(a_n)_{n=0}^{+\infty}$ into a sequence $\mathcal{I}(a)=b=(b_n)_{n=0}^{+\infty}$ whose ordinary generating function satisfies 
\begin{equation*} 
\sum\limits_{n=0}^{+\infty}b_nt^n=\frac{\sum\limits_{n=0}^{+\infty}a_nt^n}{1-t\sum\limits_{n=0}^{+\infty}a_nt^n}. 
\end{equation*}
Barbero et al. \cite{Bar} highlighted the close relation between Invert transform and complete ordinary Bell polynomials: given $g \in H_R$ and $h = \mathcal I(g)$, we have, for all $n\geq0$, that
\begin{equation}\label{hn}
h_n = B_{n+1}(g_0,g_1,g_2,...,g_n).
\end{equation}

Now these tools allow us to explicitly find the terms of  $b=a^{-1}$ for every $a\in H_{R}^{*}$.

\begin{theorem}
Let $a, b=a^{-1} \in H_R^*$ be sequences with e.g.f. $A(t)$ and $B(t)$, respectively. Then for all $n\geq 0$, we have
$$b_n = \frac{n!B_n(g_0,g_1,g_2,\ldots,g_n)}{a_0},$$
where
$$g=(g_n)_{n=0}^{+\infty}=\left(-\frac{a_{n+1}}{a_0(n+1)!}\right)_{n=0}^{+\infty}.$$
\end{theorem}
\begin{proof}
The ordinary generating function of the sequence $g$ is
$$\bar G(t) = \frac{1}{t}\left( 1 - \frac{A(t)}{a_0} \right)$$
since
$$\bar G(t) = \frac{1}{t}\left(1- \frac{1}{a_0}\left( a_0 + \sum_{n=1}^{+\infty} \frac{a_n}{n!}t^n \right) \right) = -\frac{1}{t}\sum_{n=1}^{+\infty} \frac{a_n}{a_0n!}t^n = \sum_{n=0}^{+\infty}\left( -\frac{a_{n+1}}{a_0(n+1)!} \right)t^n.$$
Moreover, considering  $b=a^{-1}$ we have
$$
B(t) = \frac{1}{A(t)}=\frac{1}{a_0(1-t\bar G(t))}=\frac{1}{a_0}\left( 1 + \frac{t\bar G(t)}{1-t\bar G(t)} \right)=\frac{1}{a_0}(1 + t\bar H(t)),$$
where  $\bar H(t)$ is  the ordinary generating function of the sequence $h = \mathcal I(g)$. Thus relation (\ref{hn}) holds and,  since $B_0(g)=1$, we obtain
$$B(t) =\frac{1}{a_0}(1 + t\bar H(t))= \frac{1}{a_0}\left( 1 + \sum_{n=0}^{+\infty} B_{n+1}(g_0,g_1,g_2,\ldots,g_n)t^{n+1} \right) = \sum_{n=0}^{+\infty} \frac{B_n(g_0,g_1,g_2,\ldots,g_n)}{a_0}t^n$$
and the thesis easily follows.  
\end{proof}

We point out that some well--studied transforms acting on sequences can be viewed in $H_{R}[[t]]$ as the product (i.e., the binomial convolution) between a suitable fixed sequence and any sequence belonging to $H_{R}$.  We present two enlightening and interesting examples.
\subsubsection*{The Binomial interpolated transform}
The \emph{Binomial interpolated transform} $\mathcal L^{(y)}$, with parameter $y \in R$, maps any sequence $a \in H_R$ into a sequence $b = \mathcal L^{(y)}(a) \in H_R$, whose terms are
$$b_n = \sum_{h=0}^n\binom{n}{h}y^{n-h}a_h.$$
For a survey and a detailed study of the action of $\mathcal{L}^{(y)}$ on recurrence sequences we refer the reader to  \cite{Bar}.
The definition of this transform  by means of the binomial convolution is straightforward. Indeed, considering the sequence
$$\lambda = (y^n)_{n=0}^{+\infty},$$
we have for any $a \in H_R$ with e.g.f. $A(t)$
$$\mathcal L^{(y)}(a) = \lambda \star a.$$
with the corresponding e.g.f. given by the product $e^{yt}A(t).$

\subsubsection*{The Boustrophedon transform}
The \emph{Boustrophedon transform} $\mathcal B$, introduced and studied in \cite{Mil}, maps any sequence $a \in H_R$, with e.g.f. $A(t)$, into a sequence $b = \mathcal B(a) \in H_R$ with e.g.f. 
$$B(t) = (\sec t + \tan t)A(t).$$
This transform is closely related to the sequence $\beta = (\beta_n)_{n=0}^{+\infty}$ of the Euler zigzag numbers (see \cite{Mil}), with e.g.f.
$$\sum_{n=0}^{+\infty} \frac{\beta_n}{n!}t^n := \sec(t) + \tan(t),$$
since for any $a \in H_R$ clearly
$$\mathcal B(a) = \beta \star a.$$

The Hurwitz series ring is strictly connected to other well--known  sequence transforms. We consider another two examples, the alternating sign transform, which is  a little bit trivial, and the  interesting Stirling transform. We also show that both are examples of $H_{R}$--\emph{automorphisms}. 
\subsubsection*{The alternating sign trasform}
The alternating sign transform $\mathcal E$ maps any sequence $a \in H_R$ into a sequence $b = \mathcal E(a) \in H_R$, whose terms are
$$b_n = (-1)^n a_n.$$
The transform $\mathcal E$ often appears in studying properties of integer sequences combined with other transforms.
Clearly we have $\mathcal{E}=\mathcal{E}^{-1}$ and it is straightforward to see that, given any $a \in H_R$ with e.g.f. $A(t)$, then $\mathcal E(a)$ has e.g.f. $A(-t)$. Moreover it is easy to verify that for all sequences $a,b \in H_{R}$ 
$$\mathcal{E}(a+b)=\mathcal{E}(a)+\mathcal{E}(b),$$
and
$$\mathcal{E}(a\star b)=\mathcal{E}(a)\star \mathcal{E}(b),$$
showing that $\mathcal{E}$ is an authomorphism of $H_{R}$.

\subsubsection*{The Stirling transform}
The \emph{Stirling transform} $\mathcal S$ maps any sequence $a \in H_R$ into a sequence $b = \mathcal S(a) \in H_R$, whose terms are
$$b_n = \sum_{h=0}^n \stirlingtwo{n}{h} a_h,$$
where $\stirlingtwo{n}{h}$ are the Stirling numbers of the second kind
(see e.g. \cite{Gra}, chapter 6, for definition and properties of Stirling numbers of first and second kind).
Some properties of this transform are exsposed in \cite{Ber}, here we
observe that $\mathcal S$ is a bijection from $H_{R}$ to itself. The inverse  $\mathcal{S}^{-1}$  maps any sequence $a \in H_R$ into a sequence $b = \mathcal S^{-1}(a) \in H_R$, whose terms are
$$b_n = \sum_{h=0}^{n}(-1)^{n-h} {n \s h} a_h,$$
where $n \s h$ are the (unsigned) Stirling numbers of the first kind.
Moreover, we recall that for all $a \in H_R$ with e.g.f. $A(t)$, then $b = \mathcal S(a)$ has e.g.f. $B(t) = A(e^t-1)$.
It is very interesting to observe that, for all $a,b \in H_{R},$  $\mathcal{S}$ obviously satisfies 
$$\mathcal S(a + b) = \mathcal S(a) + \mathcal S(b),$$
but also
$$\mathcal S(a \star b) = \mathcal S(a) \star \mathcal S(b).$$
Indeed, remembering that, by definition, ${n\brace h}=0$ when $n<k$, and that the e.g.f. of the Stirling number of the second kind is $\frac{(e^{t}-1)^n}{n!}$ (see \cite{Gra}) , if we consider the e.g.f. $S(t)$ of $\mathcal{S}(a\star b)$, we have

\begin{align*}
S(t)&=\sum_{n=0}^{+\infty}\left(\sum_{h=0}^{n}{n \brace h} \left(\sum_{j=0}^{h} \binom{h}{j} a_{j}b_{h-j}\right)\right)\frac{t^{n}}{n!}=\sum_{h=0}^{+\infty}\left(\sum_{j=0}^{h}\binom{h}{j}a_{j}b_{h-j}\right)
\sum_{h=0}^{+\infty}{n \brace h}\frac{t^{n}}{n!}=\\
&=\sum_{h=0}^{+\infty}\left(\sum_{j=0}^{h}\binom{h}{j}a_{j}b_{h-j}\right)
\frac{\left(e^t-1\right)^{n}}{n!}=\mathcal{A}\left(e^{t}-1\right)\mathcal{B}\left(e^{t}-1\right)
\end{align*}
and $\mathcal{A}\left(e^{t}-1\right)\mathcal{B}\left(e^{t}-1\right)$ clearly is the e.g.f. of $\mathcal S(a) \star \mathcal S(b)$.
Hence $\mathcal S$ is an authomorphism of $H_R$. 



\section{Special subgroups of $H_R^{*}$}

The purpose of this section is to highlight some properties of two interesting subgroups of $H_R^*$, with respect to the binomial convolution product $\star$ operation. We also study  their relationship with the transforms presented in the previous section and with other transforms which we will define in the next. 
\begin{definition}
Let us denote $U_R$ and $B_R$ the subgroups of $H_R^*$ defined as
$$U_R=\{a \in H_R^*: a_0=1\}, \quad B_R=\{a \in U_R: \mathcal E(a) = a^{-1}\}.$$
\end{definition}
We start considering the subgroup $U_R$ and observing that, for all $a\in H_{R}$, we can find sequences in $U_{R}$ closely related with $a$, obtained by prepending to $a$ a finite sequence of 1. So it is natural to consider these sequences as the images of $a$ under the iteration of the following transform.

The \emph{1--prepending} transform $\mathcal V$ maps a sequence $a=(a_0,a_1,a_2,\ldots) \in H_R$ into the sequence $b=\mathcal{V}(a) =(1,a_0,a_1,a_2,\ldots) \in U_R$. We denote  $\mathcal V^k$ as the $k$--times iteration of $\mathcal V$ and obviously $\mathcal V^k(a) \in U_R$, for all $k \geq 1$. We observe that the action of $\mathcal{V}^{k}$ on a sequence $a\in H_{R}^*$ corresponds to the $k$--th iteration of the integral operator  
$$\mathcal J(\cdot)(t) := 1 + \int_0^t(\cdot) du$$
on the e.g.f. $A(t)$ of $a$. In particular, 
$$\mathcal J^k (A)(t)= \underbrace{\mathcal J \circ ... \circ \mathcal J}_{k-\text{times}}(A)(t)=\sum_{h = 0}^{k - 1}\frac{t^h}{h!} + \sum_{h = k}^{+\infty} a_{h-k}\frac{t^{h}}{h!}=V(t),$$ 
where $V(t)$ is the e.g.f. of $\mathcal{V}^{k}(a)$, and clearly
$$V(0)=V'(0)=...=V^{(k-1)}(0)=1, \quad V^{(k)}(t) = A(t),$$
being $V^{(k)}(t)$  the $k$--th derivative of $V(t)$. 

Now, we explore some interesting properties related to the subgroup $B_R$. First of all we  characterize all the elements in $B_{R}$.
\begin{theorem}\label{br}
All the elements in $B_{R}$ corresponds to the sequences of $U_{R}$  whose e.g.f. $A(t)$ is the solution of 
\begin{equation}\label{system}
\begin{cases}
A'(t)=g(t)A(t)\\
A(0)=1
\end{cases}
\end{equation} 
where $g(t)$ is any fixed even function. Hence, if we consider the formal exponential operator $exp$ such that $exp(f(t))=\sum\limits_{n=0}^{+\infty}\frac{(f(t))^{n}}{n!}$,  the e.g.f. of $a\in B_{R}$ is $A(t)=exp(h(t))$ where $h(t)$ is an odd function.
\end{theorem}
\begin{proof}
 It is immediate to see that $a \in B_R$ if and only if $A(t)A(-t)=1$. If we differentiate this relation with respect to $t$ we obtain
 $$A'(t)A(-t)-A(t)A'(-t)=0$$
 which is equivalent to 
 $$\frac{A'(t)}{A(t)}=\frac{A'(-t)}{A(-t)}.$$
 Thus we have $g(t)=\frac{A'(t)}{A(t)}$ where, from the previous relation, $g(t)=g(-t)$, i.e. $g(t)$ is an even function, and, obviously, we must have $A(0)=1$, since  for all sequences  $a\in U_{R}$ $A(0)=a_0=1$. It is straightforward to verify that, given $g(t)$, a formal integration term by term of its power series corresponds to the series of an odd function $h(t)$, and clearly $A(t)=exp(h(t))$ satisfies (\ref{system}).
\end{proof}
The transforms $\mathcal{L}^{y},\mathcal{B},$ and $\mathcal{E}$ act on $B_{R}$  preserving the closure, as we point out in the following proposition.

\begin{proposition}
The group $B_R$ is closed with respect to the transforms $\mathcal E$, $\mathcal B$ and $\mathcal L^{(y)}$, for any $y \in R$.
\end{proposition}
\begin{proof}
By definition of $B_R$, it is immediate to check that $\mathcal E(B_R) = B_R$ (with this notation, we say that given any $a \in B_R$, then $\mathcal E(a)$ is still in $B_R$).
Given any $a \in B_R$, with e.g.f. $A(t)$, we have that $b = \mathcal L^{(y)}(a)$ has e.g.f. $e^{(yt)}A(t)$, and clearly $b \in B_R$ since 
$$\left(e^{yt}A(t)\right)\left(e^{-yt}A(-t)\right)=A(t)A(-t)=1.$$
Finally, let us recall that the Euler zig--zag numbers $\beta$ have e.g.f. $B(t) = \sec(t) + \tan(t)$ which satisfies $B(t)B(-t)=1$  as a simple calculation shows
$$B(t) - \frac{1}{B(-t)} = \frac{1+\sin(t)}{\cos(t)}-\frac{\cos(-t)}{1+\sin(-t)} = \frac{1-\sin^2(t)-\cos^2(t)}{\cos(t)(1-\sin(t))} = 0,$$
i.e., $\beta \in B_R$. Hence, given any $a \in B_R$, $\mathcal B(a) = \beta \star a \in B_R$.
\end{proof}

\begin{remark}
The group $B_R$ is not closed with respect to the transform $\mathcal S$. Indeed, if $A(t)$ is the e.g.f. of $a\in B_{R}$, the e.g.f. of $S(a)$ is $A(e^{t}-1)$, while the e.g.f. of $\varepsilon(S(a))$ is $A(e^{-t}-1)$ and in general $A(e^{-t}-1)A(e^{t}-1)\neq 1$. 
It would be interesting to characterize the group $\mathcal S(B_R)$.
\end{remark}
Let us examine in depth the structure of a sequence $a \in B_{R}$. From  the definition of $B_{R}$ and from Theorem \ref{br}, we observe that
the elements of a sequence $a \in B_R$ are constrained to severe restrictions, since the equality $\mathcal E(a) = a^{-1}$ must hold. If we pose $b = a^{-1}$ and $c = \mathcal E(a)$, we have, for instance,
$$b_0 = 1, \quad b_1 = -a_1, \quad b_2 = 2a_1^2-a_2$$
and
$$c_0 = 1, \quad c_1 = -a_1, \quad c_2 = a_2,$$
i.e., the element $a_1$ of the sequence $a$ can be arbitrary, while $a_2$ must satisfy
$$a_2 = 2a_1^2 - a_2,$$
i.e., $a_2 = a_1^2$. By continuing in this way, we can also see, e.g., that $a_3$ can be arbitrary, while $a_4 = -3a_1^4 + 4a_1a_3$. Thus, any sequence $a \in B_{R}$ is completely determined when we fix the values of $a_{2k-1}, k=1,2,\cdots$. Indeed, the following theorem shows how to evaluate the terms with even positive index as functions of the ones with odd index, by means of the partial ordinary Bell polynomials described in Definition \ref{def:bell}.

\begin{theorem} \label{thm:a2n}
Given any $a \in B_R$, we have
$$a_{2n} = (2n)! \sum_{k=0}^n \binom{\frac{1}{2}}{k}B_{n+k,2k}(x_1,...,x_{n-k+1}),\quad \forall n\geq1,$$
where $x_i=\frac{a_{2i-1}}{(2i-1)!}$, $\binom{\frac{1}{2}}{k} = \frac{\prod_{j=0}^{k-1}\left(\frac{1}{2}-j\right)}{k!}.$
\end{theorem}
\begin{proof}
Let $A(t)$ be the e.g.f. of $a$. Clearly  $A(t) = P(t) + D(t)$, where
\begin{equation}\label{pt}
P(t) = \sum_{n=0}^{+\infty} \frac{a_{2n}}{(2n)!}t^{2n}
\end{equation}
and
\begin{equation}\label{dt}
 D(t) = \sum_{n=1}^{+\infty} \frac{a_{2n-1}}{(2n-1)!}t^{2n-1} = \frac{1}{t}\sum_{n=1}^{+\infty} \frac{a_{2n-1}}{(2n-1)!}t^{2n}.
 \end{equation}
Moreover, we have
$$1 = A(t)A(-t) = (P(t) + D(t))(P(t) - D(t)) = (P(t))^2 - (D(t))^2,$$
since $A(-t) = P(-t) + D(-t) = P(t) - D(t)$ and $\mathcal E(a) = a^{-1}$. Now, observing that $P(0)=a_0=1$ and $D(0)=0$, we obtain from the formal Maclaurin power series of $(1+X)^{\frac{1}{2}}$ that
$$P(t) = \left(1 + \left(D(t)\right)^2\right)^{\frac{1}{2}} = \sum_{k=0}^{+\infty} \binom{\frac{1}{2}}{k}\left(D(t)\right)^{2k}.$$
By definition of partial ordinary Bell polynomials we have
$$\left(D(t)\right)^{2k} = \sum_{m=2k}^{+\infty} B_{m,2k}(x_1,...,x_{m-2k+1})t^{2m-2k}.$$
If we set $n=m-k$, we get
$$P(t) = \sum_{k=0}^{+\infty}\binom{\frac{1}{2}}{k}\sum_{n=k}^{+\infty} B_{n+k,2k}(x_1,...,x_{n-k+1})t^{2n} = \sum_{n=0}^{+\infty}\left( \sum_{k=0}^n \binom{\frac{1}{2}}{k}B_{n+k,2k}(x_1,...,x_{n-k+1}) \right)t^{2n}.$$
From this equality, comparing the coefficients of the respective even powers of $t$ in (\ref{pt}) we finally obtain
$$a_{2n} = (2n)! \sum_{k=0}^n \binom{\frac{1}{2}}{k}B_{n+k,2k}(x_1,...,x_{n-k+1}),\quad \forall n\geq1.$$
\end{proof}
By Definition \ref{def:bell} and observing that $B_{0,0}=1$, $B_{h,0}=0$ for $h\geq1$, we immediately have the following corollary.
\begin{corollary} \label{cor:a2n}
Given $a \in B_R$, we have
\begin{equation*}
a_{2n}=\left(2n\right)!\sum_{k=1}^{n}\binom{\frac{1}{2}}{k}\sum\limits_{\substack{i_1+i_2+\cdots+i_{n-k+1}=2k\\i_1+2i_2+\cdots+(n-k+1)i_{n-k+1}=n+k}}
\left(2k\right)!\prod\limits_{j=1}^{n-k+1}\frac{1}{i_{j}!\left((2j-1)! \right)^{i_{j}}}\prod\limits_{j=1}^{n-k+1}a_{2j-1}^{i_{j}}.\label{eq:a2nexplicit}
\end{equation*}
\end{corollary}
On the other hand, it is also possible to determine the sequences $a\in B_{R}$, with $a_2\in R^{*}$ and such that $a_2$ is a square in $R$, by fixing the terms $a_{2k},k=1,2,\cdots$ and finding the terms with odd index as functions of the ones with even index.
\begin{theorem} \label{thm:a2n1}
Given $a \in B_R$ such that $a_2 \in R^*$ and $x^2=a_2$ is solvable in $R$, we have
$$a_{2n+1} = (2n+1)! \sum_{k=0}^n a_2^{\frac{1}{2}-k} \binom{\frac{1}{2}}{k}B_{n,k}(x_1,...,x_{n-k+1}),\quad \forall n\geq0,$$
where $x_i=\frac{1}{(2i+2)!}\sum\limits_{k=0}^{i+1} \binom{2i+2}{k} a_{2k} a_{2(n-k+1)}$, and $a_{2}^{\frac{1}{2}}\in R$ is a  solution of
$x^2=a_2.$
\end{theorem}
\begin{proof}
Let $A(t)$ be the e.g.f. of $a$, with the same notations used in the proof of Theorem \ref{thm:a2n}, we have $(D(t))^2 = 1+(P(t))^2$, where
$D(t)$ and $P(t)$ as in (\ref{dt}) and in (\ref{pt}) respectively.
Since the product $P(t)\cdot P(t)$ is equal to
$$\left(P(t)\right)^2 = 1 + a_2 t^2 + \sum_{n=2}^{+\infty}\left( \sum_{k=0}^n \binom{2n}{2k} a_{2k}a_{2n-2k} \right)\frac{t^{2n}}{(2n)!},$$
we find
\begin{align*}D(t)&=ta_{2}^{\frac{1}{2}}\left(1+\sum_{n=2}^{+\infty}\frac{a_{2}^{-1}}{\left(2n\right)!}\left(\sum_{n=0}^{k}\binom{2n}{2k}a_{2k}a_{2n-2k}\right)t^{2n-2}\right)^{\frac{1}{2}}=\\
&=ta_{2}^{\frac{1}{2}}\left(1+\sum_{n=1}^{+\infty}\frac{a_{2}^{-1}}{\left(2n+2\right)!}\left(\sum_{k=0}^{n+1}\binom{2n+2}{2k}a_{2k}a_{2n-2k+2}\right)t^{2n}\right)^{\frac{1}{2}}
\end{align*}
Then, considering the formal Maclaurin series expansion of $(1+X)^{\frac{1}{2}}$ and by Definition \ref{def:bell}, we obtain
$$D(t) = \sum_{n=0}^{+\infty}\left(\sum_{k=0}^{n}a_2^{1/2-k} \binom{\frac{1}{2}}{k} B_{n,k}(x_1,...,x_{n-k+1}) \right)t^{2n+1}.$$
Now the thesis esily follows by a simple comparison of the corresponding coefficients of the odd powers of $t$  in the expansion (\ref{dt}) of $D(t).$ 

\end{proof}

\begin{remark}
When $R = \mathbb Z$, $B_{\mathbb Z}$ contains many well--known and important integer sequences. We mention here some of them as interesting examples. 
We have seen that the Euler zigzag numbers belong to $B_{\mathbb Z}$. They are listed in OEIS \cite{oeis} as A000111. Thus, all the sequences having as e.g.f. a power of $\sec(t) + \tan(t)$ are in $B_{\mathbb Z}$. \\
For instance the sequence A001250 in OEIS, whose $n$--th element is the number of alternating permutations of order $n$, has e.g.f. $(\sec(t) + \tan(t))^2$.\\
Moreover, the sequence A000667, which is the Boustrophedon transform of all--1's sequence, has e.g.f. $e^t(\sec(t) + \tan(t))$ and belongs to $B_{\mathbb Z}$. \\
Another sequence in $B_{\mathbb Z}$ is A000831, with e.g.f. $\frac{1+\tan(t)}{1-\tan(t)}$.\\
The sequences A006229 and A002017  also belong $B_{\mathbb Z}$ since they have exponential generating functions of the shape $exp({f(t)})$, with $f(t)$ odd function. Indeed, they have e.g.f. $e^{\tan(t)}$ and $e^{\sin(t)}$, respectively. \\
Thanks to Theorem \ref{thm:a2n} and Corollary \ref{cor:a2n}, we have new interesting identities connecting many sequences in OEIS. Furthermore, it is quite surprising that all these (very different) sequences satisfy the same limiting conditions.
\end{remark}

In the following, we will introduce a new transform of sequences that arises from the study of $B_R$, which will allow us to consider $U_R$ as a dynamic ultrametric space.\\
Given $a, b=a^{-1} \in U_R$, we know that 
$$\sum_{h=0}^n \binom{n}{h} a_h b_{n-h} = 0,\quad \forall n\geq 1,$$
from which it follows that
$$a_n = -\sum_{h=0}^{n-1}\binom{n}{h}a_h b_{n-h},\quad b_n = -\sum_{h=0}^{n-1}\binom{n}{h}b_h a_{n-h}.$$
If $a \in B_R$, i.e. $A(-t) = A(t)^{-1}$, then, for all $n\geq 1$, we have
$$-\sum_{h=0}^{n-1}\binom{n}{h}(-1)^ha_ha_{n-h} = \begin{cases} 0 \quad \text{if } n \text{ odd} \cr 2a_n \quad \text{if } n \text{ even}  \end{cases}.$$
Thus it is natural to define the following transform.

\begin{definition} \label{def:auto}
The \emph{autoconvolution} transform $\mathcal A$ maps a sequence $a \in H_R$ into a sequence $b = \mathcal A(a) \in H_R$, where
$$\begin{cases}  b_0 = a_0 \cr b_{2n+1} = a_{2n+1}, \quad \forall n\geq0 \cr b_{2n} = -\frac{1}{2}\sum\limits_{h=1}^{n-1}\binom{n}{h}(-1)^h a_h a_{n-h} , \quad \forall n\geq1 \end{cases}.$$
\end{definition}
The following proposition is a straightforward consequence.
\begin{proposition}
Given any $a\in H_R$, we have $a \in B_R \Leftrightarrow \mathcal A(a) = a$.
\end{proposition}

Finally, we introduce another transform strictly related to $\mathcal A$.

\begin{definition}
The transform $\mathcal U$ maps a sequence $a \in H_R$ into a sequence $\mathcal U(a)=b \in H_R$ as follows:
$$\begin{cases} b_0=a_0 \cr b_{2n+1}=a_{2n+1}, \quad \forall n\geq0 \cr b_{2n}= (2n)! \sum\limits_{k=0}^n \binom{\frac{1}{2}}{k}B_{n+k,2k}(x_1,...,x_{n-k+1}),\quad \forall n\geq1 \end{cases},$$
where $x_i=\frac{a_{2i-1}}{(2i-1)!}$.
\end{definition}

\begin{remark}
Given any sequence $a \in U_R$, the transform $\mathcal U$ produces a sequence in $B_R$ where the terms with odd index are the corresponding terms of $a$.
Clearly, we have that a sequence $a \in U_R$ is in $B_R$ if and only if $a=\mathcal U(a)$.
\end{remark}

\begin{proposition}
Given $a \in H_R$, with e.g.f. $A(t)$, then $\mathcal{U}(a)$ has e.g.f.
$$U(t) = \left( 1+ \left(\frac{A(t) - A(-t)}{2}\right)^2\right)^{\frac{1}{2}}+ \frac{A(t) - A(-t)}{2}.$$
\end{proposition}
\begin{proof}
We can write $A(t) = P(t) + D(t)$, where $P(t)$ and $D(t)$ as in (\ref{pt}) and in (\ref{dt}) respectively.
We have $A(-t) = P(t) - D(t)$ and consequently $D(t) = \frac{A(t) - A(-t)}{2}$. By Theorem \ref{thm:a2n}, the terms in the even places of $\mathcal U(a)$ have e.g.f $P(t)=\left(1 + (D(t))^2\right)^{\frac{1}{2}}$. Thus, we have
$$U(t) = \left(1 + (D(t))^2\right)^{\frac{1}{2}} + D(t).$$
\end{proof}

Given $a,b \in H_R$, let us define 
$$\delta(a,b) := 2^{-k},$$
if $a_i = b_i$, for any $0 \leq i \leq k - 1$. It is well--known that $\delta$ is an ultrametric in $H_R$. Indeed, 
\begin{itemize}
\item $\delta(a,b)=0 \Leftrightarrow a=b$,
\item $\delta(a,b)=\delta(b,a)$,
\item $\delta(a,c) \leq \max (\delta(a,b), \delta(b,c))$,
\end{itemize}
for any $a,b,c \in H_R$. Thus, $(H_R,\delta)$ is an ultrametric space.

Let us recall that we denote $H_R^{(n)}$ the ring whose elements are sequences of elements of $R$ with length $n$. Similarly, $U_R^{(n)}$ and $B_R^{(n)}$ are the subgroups of $H_R^{(n)*}$ corresponding to the subgroups $U_R$ and $B_R$ of $H_R^*$, respectively.

\begin{theorem} \label{thm:au}
Given any $a \in U_R$, we have
$$\delta(\mathcal A^n(a), \mathcal U(a)) \leq \frac{1}{2^{2(n+1)}},$$
where $\mathcal A^n = \underbrace{\mathcal A \circ ... \circ \mathcal A}_{n-\textit{times}}$.
\end{theorem}
\begin{proof}
We prove the thesis by induction.\\
Let us denote $a' = \mathcal U(a)$ and $b = \mathcal A(a)$. It is straightforward to check that
$$a' = (1, a_1, a_1^2, a_3,...), \quad b = (1, a_1, a_1^2, a_3,...).$$ 
Thus, $a'$ and $b$ coincide at least in the first 4 terms, i.e., 
$$\delta(\mathcal A(a), \mathcal U(a)) \leq \frac{1}{2^4}.$$
Now, let us suppose that given $b = \mathcal A^n(a)$, we have $\delta(\mathcal A^n(a), \mathcal U(a)) \leq \frac{1}{2^{2(n+1)}}$, i.e. $b_i = a'_i$ for all $i \leq 2n+1$ and consider $c = \mathcal A(b)$. Since $a' \in B_R$, we remember that for all $n\geq 1$ we have
$$-\sum_{h=0}^{n-1}\binom{n}{h}(-1)^ha'_ha'_{n-h} = \begin{cases} 0 \quad \text{if } n \text{ odd} \cr 2a'_n \quad \text{if } n \text{ even}  \end{cases}.$$
Thus, by Definition \ref{def:auto}, we obtain $c_i = a'_i$ for all $i \leq 2n+3$, since $(b_0,...,b_{2n+1}) = (a'_0,...,a'_{2n+1}) \in B_R^{(2n+2)}$ by inductive hypothesis. Hence, we have proved that
$$\delta(\mathcal A^{n+1}(a),\mathcal U(a))\leq \frac{1}{2^{2(n+2)}}.$$
\end{proof}

As a consequence of Theorem \ref{thm:au}, we can observe that $\mathcal A$ can be considered as an approximation of $\mathcal U$. Indeed, given a sequence $a \in U_R$, sequences $\mathcal A^n(a)$ have more elements equal to elements of $\mathcal U(a)$ for increasing values of $n$.

\begin{example}
Given $a=(a_0,a_1,a_2,a_3,a_4,a_5) \in U_R^{(6)}$, then
$$\mathcal U(a) = (1,a_1,a_1^2,a_3,4a_1a_3-3a_1^4,a_5)$$
and 
$$\mathcal A(a) = (1,a_1,a_1^2,a_3,4a_1a_3-3a_2^2,a_5).$$
Considering $\mathcal A^2$, we obtain
$$\mathcal A^2(a) = (1,a_1,a_1^2,a_3,4a_1a_3-3a_1^4,a_5) = \mathcal U(a).$$
In other words, given any sequence $a \in U_R^{(6)}$, $\mathcal A^2(a) = \mathcal U(a) \in B_R^{(6)}$, i.e., in $U_R^{(6)}$ the transforms $\mathcal A^2$ and $\mathcal U$ are identical.
\end{example}
 Frow Theorem \ref{thm:au} easily follows the next corollary.
\begin{corollary}
Given any $a \in U_R^{(2n)}$, we have
$$\mathcal U(a) = \mathcal A^{n-1}(a).$$
Moreover, for any $a \in U_R$, we have
$$\mathcal U(a) = \lim_{n\rightarrow +\infty}\mathcal A^{n}(a).$$

\end{corollary}
Clearly, if two sequences $a, b \in H_R$ coincide in the first $k$ terms, then $\mathcal A(a)$ and $\mathcal A(b)$ coincide at least in the first $k$ terms. Thus, we have the following proposition.
\begin{proposition}
Given any $a,b \in H_R$, then
$$\delta(\mathcal A(a), \mathcal A(b)) \leq \delta(a,b).$$
\end{proposition}

By the previous proposition, we have that $\mathcal A$ is a contraction mapping on the ultrametric space $(H_R, \delta)$. As a first interesting consequence, we can observe that $\mathcal A$ is a continuous function. Moreover, we have that the ultrametric group $(U_R, \star, \delta)$ with the contraction mapping $\mathcal A$ is an ultrametric dynamic space, where the set of fixed points is the subgroup $B_R$. In this way, we have found a very interesting example of ultrametric dynamic space. Ultrametric dynamics are very studied in several fields, see \cite{Pri} for a good reference about dynamics on ultrametric spaces.

\end{document}